\documentclass[11pt]{amsart}
\usepackage{amssymb}
\usepackage{amsmath}
\usepackage{amscd}
\usepackage{verbatim}
\usepackage{amssymb,amsfonts,pstricks,epsf}
\usepackage{graphicx}
\usepackage{epsfig}
\usepackage{color}
\usepackage{stmaryrd}
\usepackage{mathrsfs}
\usepackage{marginnote}
\usepackage{geometry}
\usepackage{mathtools}
\newtheorem{theorem}{Theorem}

\newtheorem{corollary}[theorem]{Corollary}
\newtheorem{remark}[theorem]{Remark}

\newtheorem{innercustomthm}{Lemma}
\newenvironment{customthm}[1]
  {\renewcommand\theinnercustomthm{#1}\innercustomthm}
  {\endinnercustomthm}

\newcommand{\R}{\mathbb R}

\newcommand{\dint}{\displaystyle\int}
\definecolor{plum}{rgb}{1.0, 0.0, 1.0}

\begin{document}

\title[Sums of eigenvalues of free plates]{Estimates for sums of eigenvalues of the free plate via the Fourier transform}

\author{Barbara Brandolini, Francesco Chiacchio, and Jeffrey J. Langford}

\address{Department of Mathematics, Bucknell University, Lewisburg, Pennsylvania 17837}

\email{jeffrey.langford@bucknell.edu}

\address{Dipartimento di Matematica e Applicazioni ``R. Caccioppoli'', Universit\`a degli Studi di Napoli Federico II, Napoli, Italy}

\email{brandolini@unina.it}
\email{francesco.chiacchio@unina.it}

\date{\today}

\begin{abstract}
Using the Fourier transform, we obtain upper bounds for sums of eigenvalues of the free plate. 
\end{abstract}



\maketitle

\section{Introduction and main results}
For a bounded domain $\Omega \subset \mathbb{R}^n$ with smooth boundary, the frequencies and modes of vibration for a free membrane of shape $\Omega$ satisfy the Neumann eigenvalue problem
\begin{equation}\label{eqn:memprob}
\begin{cases}
-\Delta u=\mu u & \textup{in }\Omega,\\
\frac{\partial u}{\partial \nu}=0 & \textup{on }\partial \Omega,
\end{cases}
\end{equation}
where $\Delta$ is the Laplace operator and $\frac{\partial u}{\partial \nu}$ is the outer normal derivative. It is well known that the free membrane problem admits a spectrum of eigenvalues
\[
0=\mu_1(\Omega)<\mu_2(\Omega)\leq \mu_3(\Omega)\leq \cdots \to +\infty.
\]
Estimates for  the eigenvalues $\{\mu_j(\Omega)\}$ and for their sums  in terms of the geometry of $\Omega$ have been obtained by many authors (see  \cite{BCDL, BCT1, BCT2, D1,D2, FK, K1, K2, L, LS, LT}, for instance; see also \cite{HS, L, LY, M} and the references therein for analogous estimates for the fixed membrane and \cite{AB,AL,CW1, CW2, N,T,YY} for analogous estimates for the clamped plate).
For the purposes of this note, we simply recall the following estimate of Kr\"oger \cite{K1} for sums of eigenvalues:
\begin{equation}\label{ineq:Kroger1}
\sum_{j=1}^m\mu_j(\Omega)\leq (2\pi)^2\left(\frac{n}{n+2}\right)\left(\frac{1}{\omega_n|\Omega|}\right)^{\frac{2}{n}}m^{\frac{n+2}{n}},\quad m\geq 1,
\end{equation}
and also the consequential estimate for eigenvalues:
\begin{equation}\label{ineq:Kroger2}
\mu_{m+1}(\Omega)\leq (2\pi)^2 \left(\frac{n+2}{2\omega_n|\Omega|}\right)^{\frac{2}{n}}m^{\frac{2}{n}},\quad m\geq 0.
\end{equation}
Here  $|\Omega|$ denotes the volume of $\Omega$ and $\omega_n$ denotes the volume of the unit ball in $\mathbb{R}^n$.  

The goal of the present paper is to establish analogous estimates to \eqref{ineq:Kroger1} and \eqref{ineq:Kroger2} for the free plate problem.
With $\Omega$ as above, the frequencies and modes of vibration for a free plate of shape $\Omega$ are governed by the eigenvalue problem
\begin{equation}\label{eqn:freeplprob}
\begin{cases}
\Delta^2 u -\tau \Delta u=\Lambda u & \textup{in }\Omega,\\
\frac{\partial^2 u}{\partial \nu^2}=0 & \textup{on }\partial \Omega,\\
\tau \frac{\partial u}{\partial \nu}-\textup{div}_{\partial \Omega}\left(\textup{Proj}_{T_x(\partial \Omega)}[(D^2u)\nu]\right)-\frac{\partial \Delta u}{\partial \nu} =0 & \textup{on }\partial \Omega,
\end{cases}
\end{equation}
where $\Delta^2u=\Delta(\Delta u)$ is the bilaplace operator, $\tau\in \R$, $\textup{div}_{\partial \Omega}$ denotes the divergence operator for the surface $\partial \Omega$, $D^2u$ denotes the Hessian matrix, and $\textup{Proj}_{T_x(\partial \Omega)}$ denotes the orthogonal projection of a vector from $T_x\mathbb{R}^n$ onto the tangent space $T_x(\partial \Omega)$. 
  In this paper, we study problem \eqref{eqn:freeplprob} when the parameter $\tau\geq0$; in this case, the eigenvalue problem for the free plate exhibits a nonnegative spectrum (see \cite{C1,C2})
\begin{equation*}\label{EvaluesFP}
0= \Lambda_1(\Omega) \le \Lambda_2(\Omega)\le \Lambda_3(\Omega)\leq \cdots \to +\infty.
\end{equation*}
   We observe that constants are solutions to problem \eqref{eqn:freeplprob} with eigenvalue zero for any parameter $\tau$.  If $\tau=0$, the coordinate functions $x_1,\ldots,x_n$ are additional solutions with eigenvalue zero, and so the lowest eigenvalue is at least $(n+1)$-fold degenerate. When $\tau>0$, we have a \emph{free plate under tension} and here $\Lambda_2(\Omega)>0$ (see \cite{C1}).

Since problem \eqref{eqn:freeplprob} is the ``plate analogue'' of problem \eqref{eqn:memprob} (see Section \ref{sect:freebc} for further discussion), it is not surprising that the spectra of the two problems share similar properties. For instance, a classical result of Szeg\H{o} \cite{S} and Weinberger \cite{W} states that among all domains with fixed volume, the lowest nonzero Neumann eigenvalue $\mu_2(\Omega)$ is maximized by a ball. In a relatively recent and analogous result, Chasman has shown in \cite{C2} that among all domains with prescribed volume, the first nonzero eigenvalue $\Lambda_2(\Omega)$  for a free plate under tension is maximized by a ball. The results of our paper shed additional light on the connection between problems \eqref{eqn:memprob} and \eqref{eqn:freeplprob}. More precisely, we prove:

\begin{theorem}\label{thm:esumest}
Let $\Omega \subset \mathbb{R}^n$ be a smooth bounded domain with $\{\Lambda_j(\Omega)\}$ the eigenvalues of the free plate problem \eqref{eqn:freeplprob}.  If $\tau \geq 0$, then
\begin{equation*}
\sum_{j=1}^m \Lambda_j(\Omega) \leq (2\pi)^4\left(\frac{n}{n+4}\right)\left(\frac{1}{\omega_n|\Omega|}\right)^{\frac{4}{n}}m^{\frac{n+4}{4}}+\tau(2\pi)^2\left(\frac{n}{n+2}\right)\left(\frac{1}{\omega_n|\Omega|}\right)^{\frac{2}{n}}m^{\frac{n+2}{n}},\qquad m\geq 1.
\end{equation*}
\end{theorem}
As a consequence of Theorem \ref{thm:esumest}, we obtain the following eigenvalue estimates:

\begin{corollary}\label{cor:eest}
Let $\Omega$ and $\{\Lambda_j(\Omega)\}$ be as in Theorem \ref{thm:esumest}. If $\tau=0$, then
\begin{equation*}
\Lambda_{m+1}(\Omega)\leq (2\pi)^4\left(\frac{n+4}{4\omega_n|\Omega|}\right)^{\frac{4}{n}}m^{\frac{4}{n}},\qquad  m\geq 0,
\end{equation*}
while  when $\tau >0$, we have

$$
\Lambda_{m+1}(\Omega)\le \min_{r >  2\pi\left(\frac{m}{\omega_n|\Omega|}\right)^{\frac{1}{n}} }\frac{n\omega_n|\Omega| \left(\frac{r^{n+4}}{n+4}+\tau \frac{r^{n+2}}{n+2}\right)}{\omega_n|\Omega|r^n-m(2\pi)^n}, \qquad m \ge 0.
$$

\end{corollary}
The remainder of this note is divided into two sections. Section \ref{sect:freebc} presents a discussion of the boundary conditions of the free plate problem \eqref{eqn:freeplprob} while Section \ref{sect:mainres} presents proofs of the main results and further consequences.
\section{Free boundary conditions}\label{sect:freebc}
To better understand the boundary conditions appearing in the plate problem \eqref{eqn:freeplprob}, we return our attention to the membrane problem \eqref{eqn:memprob}. The bilinear form for the membrane problem is given by
\[
a(u,v)=\int_{\Omega}\nabla u \cdot \nabla v\,dx,\qquad u,v\in H^1(\Omega).
\]
To say that $u\in H^1(\Omega)$ is a weak solution to problem \eqref{eqn:memprob} means that
\begin{equation}\label{eqn:weaksolnmem}
a(u,v)=\mu \int_{\Omega}uv\,dx
\end{equation}
for all functions $v\in H^1(\Omega)$. In particular, if $u$ is a weak solution and $u,v\in C^{\infty}(\overline{\Omega})$, integrating by parts transforms equation \eqref{eqn:weaksolnmem} into
\begin{equation}\label{eqn:ibpmem}
-\int_{\Omega}(\Delta u)v\,dx+\int_{\partial \Omega}\frac{\partial u}{\partial \nu}v\,dS=\mu\int_{\Omega}uv\,dx.
\end{equation}
Since \eqref{eqn:ibpmem} holds for functions $v=0$ along the boundary $\partial \Omega$, we see that $-\Delta u=\mu u$ in $\Omega$ in the classical sense. Hence \eqref{eqn:ibpmem} becomes
\[
\int_{\partial \Omega}\frac{\partial u}{\partial \nu}v\,dS=0
\]
for all functions $v\in C^{\infty}(\overline{\Omega})$, and we likewise deduce that $\frac{\partial u}{\partial \nu}=0$ in the classical sense along the boundary $\partial \Omega$. The term ``free'' in problem \eqref{eqn:memprob} comes the weak formulation \eqref{eqn:weaksolnmem}, where functions in the space $H^1(\Omega)$ have no prescribed behavior on the boundary. The boundary condition $\frac{\partial u}{\partial \nu}=0$ arises naturally from the weak formulation of our eigenvalue problem.

The bilinear form associated to the free plate problem \eqref{eqn:freeplprob} is given analogously by
\begin{equation*}\label{eqn:plateform}
A(u,v)=\sum_{j,k=1}^n\int_{\Omega}u_{x_jx_k}v_{x_jx_k}\,dx +\tau  \int_\Omega \nabla u \cdot \nabla v\,dx,\qquad u,v\in H^2(\Omega).
\end{equation*}
We say that $u$ is a weak solution to problem \eqref{eqn:freeplprob} if
\begin{equation*}
A(u,v)=\Lambda \int_{\Omega}uv\,dx
\end{equation*}
for each $v\in H^2(\Omega)$. Thus, if $u$ is a weak solution and $u,v\in C^{\infty}(\overline{\Omega})$, integration by parts transforms the above equation into

\begin{eqnarray}
A(u,v)&=& \int_{\Omega}(\Delta^2u-\tau \Delta u)v\,dx +\int_{\partial \Omega}\frac{\partial^2u}{\partial \nu^2}\frac{\partial v}{\partial \nu}\,dS \label{eqn:weakformplate}
\\
&&\> +\int_{\partial \Omega}\left(\tau \dfrac{\partial u}{\partial \nu}-\textup{div}_{\partial \Omega}\left(\textup{Proj}_{T_x(\partial \Omega)}[(D^2u)\nu]\right)-\frac{\partial \Delta u}{\partial \nu}\right)v\,dS  \notag
\\
&=&\Lambda \int_{\Omega}uv\,dx . \notag
\end{eqnarray}
For the details of this calculation, see \cite{C1}. Taking $v\in C_c^{\infty}(\Omega)$ to be a test function, we see that $\Delta^2u-\tau \Delta u=\Lambda u$ in $\Omega$ in the classical sense. Thus, equation \eqref{eqn:weakformplate} becomes
\begin{equation}\label{eqn:bdryplate}
0=\int_{\partial \Omega}\frac{\partial^2u}{\partial \nu^2}\frac{\partial v}{\partial \nu}\,dS+\int_{\partial \Omega}\left(\tau \dfrac{\partial u}{\partial \nu}-\textup{div}_{\partial \Omega}\left(\textup{Proj}_{T_x(\partial \Omega)}[(D^2u)\nu]\right)-\frac{\partial \Delta u}{\partial \nu}\right)v\,dS.
\end{equation}
Observe that any smooth function $v\in C^{\infty}(\partial \Omega)$ can be extended to $C^{\infty}(\overline{\Omega})$ with $\frac{\partial v}{\partial \nu}=0$ along the boundary $\partial \Omega$. Such an extension can be constructed, for example, by first extending $v$ to be constant along the inner normal direction (for a small fixed distance) and then using a $C^{\infty}$ bump function to extend $v$ to the rest of $\overline{\Omega}$. This observation implies that each boundary integral in \eqref{eqn:bdryplate} vanishes separately, and arguing as before, we have that
\[
\frac{\partial^2u}{\partial \nu^2}=0 \quad \textup{and}\quad  \tau \dfrac{\partial u}{\partial \nu} -\textup{div}_{\partial \Omega}\left(\textup{Proj}_{T_x(\partial \Omega)}[(D^2u)\nu]\right)-\frac{\partial \Delta u}{\partial \nu}=0
\]
in the classical sense along $\partial \Omega$.

\section{Main results}\label{sect:mainres}
We begin this section with a proof of Theorem \ref{thm:esumest}.

\begin{proof}[Proof of Theorem \ref{thm:esumest}]
We use some of the ideas contained in \cite{K1}. Let $\phi_1,\ldots,\phi_m$ denote an orthonormal set of eigenfunctions for $\Lambda_1,\ldots,\Lambda_m$ and define
\[
\Phi(x,y)=\sum_{j=1}^m\phi_j(x)\phi_j(y),\qquad x,y\in \Omega.
\]
Let $\widehat{\Phi}(z,y)$ denote the Fourier transform of $\Phi$ in the variable $x$, so that
\[
\widehat{\Phi}(z,y)=\frac{1}{(2\pi)^{\frac{n}{2}}}\int_{\Omega}\Phi(x,y)e^{ix\cdot z}\,dx.
\]
Observe that
\[
(2\pi)^{\frac{n}{2}}\widehat{\Phi}(z,y)=\sum_{j=1}^m\phi_j(y)\int_{\Omega}\phi_j(x)e^{ix\cdot z}\,dx
\]
is the orthogonal projection of the function
\[
h_z(x)=e^{ix\cdot z}
\]
onto the subspace of $L^2(\Omega)$ spanned by $\phi_1,\ldots,\phi_m$. Using $\rho(z,y)=h_z(y)-(2\pi)^{\frac{n}{2}}\widehat\Phi(z,y)$ as a trial function in the Rayleigh quotient for $\Lambda_{m+1}$, we have that
\[
\Lambda_{m+1}(\Omega)\leq \frac{ \displaystyle \sum_{j,k=1}^n \int_{\Omega}\left|\rho(z,y)_{y_jy_k}\right|^2\,dy\,dz
 +\tau \sum_{j=1}^n \int_\Omega \left|\rho(z,y)_{y_j}\right|^2\,dy\,dz}{\dint_{\Omega}\left|\rho(z,y)\right|^2\,dy}.
\]
Multiplying both sides of the above inequality by the denominator and integrating over $B_r=\{z\in \R^n:|z|<r\}$, we see that
\begin{align}\label{eqn:RQforplate}
\Lambda_{m+1}(\Omega) &\leq 
 \inf_r \left\{\frac{  \displaystyle \sum_{j,k=1}^n\dint_{B_r}\int_{\Omega}\left|\rho(z,y)_{y_jy_k}\right|^2\,dy\,dz
 +\tau \sum_{j=1}^n \dint_{B_r}\int_\Omega \left|\rho(z,y)_{y_j}\right|^2\,dy\,dz
}{\dint_{B_r}\dint_{\Omega}\left|\rho(z,y)\right|^2\,dy\,dz}
\right\}
\\
&=\inf_r\left\{\frac{N}{D} \right\}, \notag
\end{align}
where the $\inf$ is taken over  $r>2\pi\left(\frac{m}{\omega_n|\Omega|}\right)^{\frac{1}{n}}$.

We first simplify the denominator $D$ in the formula above. We observe that
\[
D=I_1+I_2+I_3,
\]
where
\begin{align*}
I_1&=\int_{B_r}\int_{\Omega}|h_z(y)|^2\,dy\,dz,\\
I_2&=-2(2\pi)^{\frac{n}{2}}\textup{Re}\left\{\int_{B_r}\int_{\Omega}h_z(y)\overline{\widehat{\Phi}(z,y)}\,dy\,dz\right\},\\
I_3&=(2\pi)^n\int_{B_r}\int_{\Omega}|\widehat{\Phi}(z,y)|^2\,dy\,dz.
\end{align*}
We evaluate each integral separately. Since $|h_z(y)|=1$, we have
\[
I_1=\omega_n|\Omega|r^n.
\]
Noting
\begin{equation}\label{eqn:simpexp}
\widehat{\Phi}(z,y)=\sum_{j=1}^m\phi_j(y)\widehat{\phi}_j(z),
\end{equation} 
we have
\begin{align*}
I_2&=-2(2\pi)^{\frac{n}{2}}\textup{Re}\left\{\sum_{j=1}^m\int_{B_r}\int_{\Omega}e^{iy\cdot z}\phi_j(y)\overline{\widehat{\phi}_j(z)}\,dy\,dz\right\}\\
&=-2(2\pi)^n\sum_{j=1}^m\int_{B_r}|\widehat{\phi}_j(z)|^2\,dz.
\end{align*}
Invoking \eqref{eqn:simpexp} again, the final denominator term simplifies to
\begin{align*}
I_3&=(2\pi)^n\int_{B_r}\int_{\Omega}|\widehat{\Phi}(z,y)|^2\,dy\,dz\\
&=(2\pi)^n\sum_{j,l=1}^m\int_{B_r}\int_{\Omega}\phi_j(y)\phi_l(y)\widehat{\phi}_j(z)\overline{\widehat{\phi}_l(z)}\,dy\,dz\\
&=(2\pi)^n\sum_{j=1}^m\int_{B_r}|\widehat{\phi}_j(z)|^2\,dz.
\end{align*}
Thus,
\begin{equation}\label{eqn:D}
D=\omega_n|\Omega|r^n-(2\pi)^n\sum_{j=1}^m\int_{B_r}|\widehat{\phi}_j(z)|^2\,dz.
\end{equation}

We next turn our attention to the numerator of \eqref{eqn:RQforplate}. Observe that
\[
N=J_1+J_2+J_3,
\]
where
\begin{align*}
J_1&=\sum_{j,k=1}^n\int_{B_r}\int_{\Omega}|h_z(y)_{y_jy_k}|^2\,dy\,dz
 + \tau \sum_{j=1}^n \int_{B_r}\int_{\Omega} |h_z(y)_{y_j}|^2\,dy\,dz  ,\\
J_2&=-2(2\pi)^{\frac{n}{2}}\textup{Re}\left\{\sum_{j,k=1}^n\int_{B_r}\int_{\Omega}h_z(y)_{y_jy_k}\overline{\widehat{\Phi}(z,y)_{y_jy_k}}\,dy\,dz
+\tau \sum_{j=1}^n \int_{B_r}\int_\Omega h_z(y)_{y_j}\overline{\widehat{\Phi}(z,y)_{y_j}}\,dy\,dz\right\},\\
J_3&=(2\pi)^n\left\{\sum_{j,k=1}^n\int_{B_r}\int_{\Omega}\left|  \hat{\Phi}(z,y)_{y_jy_k}\right|^2\,dy\,dz
+\tau \sum_{j=1}^n \int_{B_r}\int_\Omega  \left| \hat{\Phi}(z,y)_{y_j}\right|^2\,dy\,dz\right\}.
\end{align*}
Since $|h_z(y)_{y_j}|=|z_j|$ and $|h_z(y)_{y_jy_k}|=|z_j||z_k|$, we have
\[
J_1=\int_{B_r}\int_{\Omega} \left(|z|^4+\tau |z|^2\right)\,dy\,dz=n\omega_n|\Omega| \left(\frac{r^{n+4}}{n+4} +\tau \frac{r^{n+2}}{n+2}\right).
\]
To compute $J_2$, we combine identity \eqref{eqn:simpexp}
with the integration by parts formula in \eqref{eqn:weakformplate} to deduce
\begin{align*}
J_2&=-2(2\pi)^{\frac{n}{2}}\textup{Re}\left\{\int_{B_r}\int_{\Omega}h_z(y)\overline{\Delta^2_y\widehat{\Phi}(z,y)}\,dy\,dz -\tau \int_{B_r}\int_\Omega h_z(y)\overline{\Delta_y\widehat{\Phi}(y,z)}\,dy\,dz\right\}\\
&=-2(2\pi)^{\frac{n}{2}}\textup{Re}\left\{\sum_{j=1}^m\Lambda_j\int_{B_r}\int_{\Omega}e^{iy\cdot z}\phi_j(y)\overline{\widehat{\phi}_j(z)}\,dy\,dz\right \}\\
&=-2(2\pi)^n\sum_{j=1}^m\Lambda_j\int_{B_r}|\widehat{\phi}_j(z)|^2\,dz.
\end{align*}
We finally compute $J_3$ again  using \eqref{eqn:weakformplate}:
\begin{align*}
J_3&=(2\pi)^n\sum_{j,k=1}^n\int_{B_r}\int_{\Omega}\left(\sum_{l_1=1}^m\phi_{l_1}(y)_{y_jy_k}\widehat{\phi}_{l_1}(z)\right)\left(\sum_{l_2=1}^m\phi_{l_2}(y)_{y_jy_k}\overline{\widehat{\phi}_{l_2}(z)}\right)\,dy\,dz\\
&\quad +(2\pi)^n\tau \sum_{j=1}^n\int_{B_r}\int_{\Omega}\left(\sum_{l_1=1}^m\phi_{l_1}(y)_{y_j}\widehat{\phi}_{l_1}(z)\right)\left(\sum_{l_2=1}^m\phi_{l_2}(y)_{y_j}\overline{\widehat{\phi}_{l_2}(z)}\right)\,dy\,dz
\\
&=(2\pi)^n\sum_{j,k=1}^n\sum_{l_1,l_2=1}^m\int_{B_r}\int_{\Omega}\phi_{l_1}(y)_{y_jy_k}\widehat{\phi}_{l_1}(z)\phi_{l_2}(y)_{y_jy_k}\overline{\widehat{\phi}_{l_2}(z)}\,dy\,dz
\\
&\quad  +(2\pi)^n\tau \sum_{j=1}^n\sum_{l_1,l_2=1}^m\int_{B_r}\int_{\Omega}\phi_{l_1}(y)_{y_j}\widehat{\phi}_{l_1}(z)\phi_{l_2}(y)_{y_j}\overline{\widehat{\phi}_{l_2}(z)}\,dy\,dz
\\
&=(2\pi)^n\sum_{l_1,l_2=1}^m\Lambda_{l_1}\int_{B_r}\int_{\Omega}\phi_{l_1}(y)\widehat{\phi}_{l_1}(z)\phi_{l_2}(y)\overline{\widehat{\phi}_{l_2}(z)}\,dy\,dz\\
&=(2\pi)^n\sum_{l_1=1}^m\Lambda_{l_1}\int_{B_r}|\widehat{\phi}_{l_1}(z)|^2\,dz.
\end{align*}
We conclude that the numerator in \eqref{eqn:RQforplate} simplifies to
\[
N=n\omega_n|\Omega|  \left(\frac{r^{n+4}}{n+4}+\tau \frac{r^{n+2}}{n+2}\right)-(2\pi)^n\sum_{j=1}^m\Lambda_j\int_{B_r}|\widehat{\phi}_j(z)|^2\,dz.
\]
Combining the expression above for $N$ with the expression for $D$ in \eqref{eqn:D}, we see that \eqref{eqn:RQforplate} becomes
\begin{equation}\label{ineq:lambd}
\displaystyle \Lambda_{m+1}(\Omega)\leq \inf_r \left \{ \dfrac{\displaystyle \frac{n\omega_n|\Omega|}{(2\pi)^n} \left(\frac{r^{n+4}}{n+4}+\tau \frac{r^{n+2}}{n+2}\right)- \sum_{j=1}^m\Lambda_j\int_{B_r}|\widehat{\phi}_j(z)|^2\,dz}{\displaystyle \frac{\omega_n|\Omega|r^n}{(2\pi)^n}- \sum_{j=1}^m\int_{B_r}|\widehat{\phi}_j(z)|^2\,dz} \right \},
\end{equation}
where we remind the reader that the $\inf$ is taken over $r>2\pi\left(\frac{m}{\omega_n|\Omega|}\right)^{\frac{1}{n}}$. By Plancherel's Theorem,
\begin{equation}\label{ineq:planch}
\int_{B_r}|\widehat{\phi}_j(z)|^2\,dz\leq 1
\end{equation}
for each $j$.  Moreover, since $\tau\geq0$, all the eigenvalues $\Lambda_j$ are nonnegative. Hence we may apply Lemma \ref{lem:appendix} in the Appendix to deduce
\[
\sum_{j=1}^m\Lambda_j(\Omega)\leq  \frac{n\omega_n|\Omega|}{(2\pi)^n} \left(\frac{r^{n+4}}{n+4}+\tau \frac{r^{n+2}}{n+2}\right),\qquad r>2\pi\left(\frac{m}{\omega_n|\Omega|}\right)^{\frac{1}{n}}.
\]
Letting $r \to 2\pi\left(\frac{m}{\omega_n|\Omega|}\right)^{\frac{1}{n}}$ gives the result.
\end{proof}

We next establish the estimate of Corollary \ref{cor:eest}.

\begin{proof}[Proof of Corollary \ref{cor:eest}]
We return our attention to the estimate of \eqref{ineq:lambd}. Combining with \eqref{ineq:planch} we deduce
\begin{equation}\label{ineq:newbdlam}
\Lambda_{m+1}(\Omega)\leq \frac{n\omega_n|\Omega| \left(\frac{r^{n+4}}{n+4}+\tau \frac{r^{n+2}}{n+2}\right)}{\omega_n|\Omega|r^n-m(2\pi)^n}=F(r), \qquad r>2\pi\left(\frac{m}{\omega_n|\Omega|}\right)^{\frac{1}{n}}.
\end{equation}
Since 
$$
\lim_{r \to 2\pi\left(\frac{m}{\omega_n|\Omega|}\right)^{\frac{1}{n}}}F(r)=\lim_{r \to +\infty} F(r)=+\infty,
$$
our   first claim immediately follows.

\noindent In the case $\tau=0$, it is easy to check that the derivative $F'(r)$ vanishes precisely when
\[
\left(\omega_n|\Omega|r^n-m(2\pi)^n\right)\left(n\omega_n|\Omega|r^{n+3}\right)-\left(n\omega_n|\Omega|\frac{r^{n+4}}{n+4}\right)\left(n\omega_n|\Omega|r^{n-1}\right)=0
\]
and solving this equation for $r$ gives
\[
r=2\pi\left(\frac{m(n+4)}{4\omega_n|\Omega|}\right)^{\frac{1}{n}}.
\]
Substituting this value of $r$ into \eqref{ineq:newbdlam} gives the result.
\end{proof}
\begin{remark}
 We make two observations when the parameter $\tau=0$. First, our proof of Corollary \ref{cor:eest} gives an alternative and elementary proof of Corollary $3.3$ from \cite{L} for the case $l=2$ without appealing to trace inequalities for convex functions of operators.   Second, if $\Lambda_{M+1}(\Omega)$ denotes the lowest nonzero free plate eigenvalue, then the estimate of Corollary \ref{cor:eest} shows
\[
C(n,|\Omega|)\sum_{m=M}^{\infty}\frac{1}{m^{\frac{4}{n}}} \leq \sum_{m=M}^{\infty}\frac{1}{\Lambda_{m+1}(\Omega)},
\]
where $C(n,|\Omega|)$ is a positive constant that depends on the dimension and volume of $\Omega$. Thus, the sum of the reciprocals of the  nonzero eigenvalues for the free plate problem diverges when the dimension $n$ is at least $4$.
\end{remark}

\section*{Appendix}
In this section we establish a lemma used in the proof of Theorem \ref{thm:esumest}. This lemma appears in \cite{LT}; we provide a proof so that our paper is self contained.
\begin{customthm}{A1}\label{lem:appendix}
Say $0\leq \Lambda_1\leq \Lambda_2\leq\cdots \leq \Lambda_{m+1}$ are such that
\begin{equation}\label{ineq:lemma}
\Lambda_{m+1}\leq \frac{a-\sum_{j=1}^m\Lambda_jc_j}{b-\sum_{j=1}^mc_j},
\end{equation}
where $a,b,c, c_j$ are positive numbers with $c_j\leq c$. If $b>mc$, then
\[
c\sum_{j=1}^m\Lambda_j \leq a.
\]
\end{customthm}

\begin{proof}
Inequality \eqref{ineq:lemma} becomes 
\[
\Lambda_{m+1}\left(b-\sum_{j=1}^mc_j\right)=a-\sum_{j=1}^m\Lambda_jc_j
\]
and rearranging terms we have
\[
\Lambda_{m+1}b-a=\sum_{j=1}^m(\Lambda_{m+1}-\Lambda_j)c_j\leq c\sum_{j=1}^m(\Lambda_{m+1}-\Lambda_j).
\]
Solving the above inequality for $c\sum_{j=1}^m\Lambda_j$ we see
\[
c\sum_{j=1}^m\Lambda_j\leq a+(mc-b)\Lambda_{m+1}.
\]
The result now follows from the assumption $b>mc$.
\end{proof}

\end{document}